\documentclass[11pt,a4paper]{amsart}

\usepackage{amsmath,amssymb,mathrsfs,enumitem,amsmath,amsthm,bbold}
\usepackage[mathscr]{euscript}
\usepackage{color}
\usepackage{hyperref}

\usepackage[all]{xy}
\definecolor{Chocolat}{rgb}{0.36, 0.2, 0.09}
\definecolor{BleuTresFonce}{rgb}{0.215, 0.215, 0.36}
\definecolor{EgyptianBlue}{rgb}{0.06, 0.2, 0.65}

\usepackage{fourier}

\newtheorem*{itheorem}{Theorem}
\newtheorem{theorem}{Theorem}[section]
\newtheorem{corollary}[theorem]{Corollary}

\newtheorem{proposition}[theorem]{Proposition}

\theoremstyle{definition}

\newtheorem{remark}[theorem]{Remark}
\newtheorem{definition}[theorem]{Definition}

\newcommand{\ac}{\scriptstyle \text{\rm !`}}

\DeclareMathAlphabet{\pazocal}{OMS}{zplm}{m}{n}

\def\calO{\pazocal{O}}
\def\calP{\pazocal{P}}

\DeclareMathAlphabet{\mathbbold}{U}{bbold}{m}{n}

\def\k{\mathbbold{k}}

\makeatletter
\@namedef{subjclassname@2020}{%
  \textup{2020} Mathematics Subject Classification}
\makeatother

\begin{document}

\title{Associator dependent algebras and Koszul duality}

\author{Murray Bremner}

\address{Department of Mathematics and Statistics, University of Saskatchewan, 142 McLean Hall, 106 Wiggins Road, Saskatoon, SK  S7N 5E6, Canada}

\email{bremner@math.usask.ca}

\author{Vladimir Dotsenko}

\address{ 
Institut de Recherche Math\'ematique Avanc\'ee, UMR 7501, Universit\'e de Strasbourg et CNRS, 7 rue Ren\'e-Descartes, 67000 Strasbourg CEDEX, France}

\email{vdotsenko@unistra.fr}

\date{}

\dedicatory{To the memory of Erwin Kleinfeld (1927-2022)}


\begin{abstract}
We resolve a ten year old open question of Loday of describing Koszul operads that act on the algebra of octonions. In fact, we obtain the answer by  solving a more general classification problem: we find all Koszul operads among those encoding associator dependent algebras. 
\end{abstract}

\maketitle


\section{Introduction}

In the context of nonassociative algebras, many interesting classes of algebras arise when one imposes linear dependency conditions for the 
associators $(a_1,a_2,a_3)=(a_1a_2)a_3-a_1(a_2a_3)$ of six permutations of the given three elements, that is
 \[
\sum_{\sigma\in S_3} x_\sigma (a_{\sigma(1)},a_{\sigma(2)},a_{\sigma(3)})=0
 \]
for certain coefficients $x_\sigma$, $\sigma\in S_3$. Various celebrated classes of nonassociative algebras, such as right-symmetric algebras, alternative algebras, flexible algebras, Lie-admissible algebras and power-associative algebras are examples of that sort. In the language of varieties of algebras, these are identities of degree three. However, from the equivalent point of view of the operad theory, these identities are \emph{quadratic} (since the product operation is used twice), and as such are within the scope of the Koszul duality theory. In this paper we classify those of operads of associator dependent algebras that are Koszul, meaning that they have particularly nice homological properties, and the deformation complexes of algebras over these operads admit simple descriptions. Our original motivation came from a slightly more narrow classification problem. In \cite{MR3013086}, Jean-Louis Loday formulated a question of finding a ``small'' Koszul operad that acts on the algebra of octonions. That algebra is the only nonassociative normed division algebra over the real numbers (by a theorem of Hurwitz \cite{MR1512117}, the others, up to isomorphism, are two fields $\mathbb{R}$ and $\mathbb{C}$, and the noncommutative algebra of quaternions $\mathbb{H}$). For octonions, the associator is antisymmetric in its arguments, in other words, they form an alternative algebra. However, Dzhumadildaev and Zusmanovich established in \cite{MR2821385} that the operad of alternative algebras is not Koszul, so the only immediate candidate is the rather uninteresting and ``big'' magmatic operad, hence the question of Loday. Our classification result leads to a complete answer to that question.

Our strategy is to reduce this classification problem to the same problem for the Koszul dual operads. For each operad of associator dependent algebras, its Koszul dual operad is a quotient of the associative operad by the operadic ideal generated by several ternary operations. In a more classical language, those operads describe  varieties of associative PI-algebras satisfying identities of degree three. Those quotient operads are easy to work with (in fact, a lot of them describe algebras whose identities include nilpotence of certain index, so the corresponding operads are supported at a finite number of possible arities), allowing us to apply various known methods for proving and disproving the Koszul property. Most of them are relatively easy to handle, with one notable exception that requires a really intricate argument, see Proposition \ref{prop:111}. In that case, an interesting by-product is the following combinatorial statement that we believe to be extremely hard to establish by classical methods of varieties of algebras: for each value of the parameter $(\alpha : \beta)\in\mathbb{P}^1$, the multilinear part of the relatively free $n$-generated algebra in the variety determined by the identity
\begin{multline*}
\beta((a_1,a_2,a_3)+(a_2,a_1,a_3))+(\alpha-\beta)((a_1,a_3,a_2)+(a_2,a_3,a_1))\\ -\alpha((a_3,a_1,a_2)+(a_3,a_2,a_1))=0
\end{multline*}
is of dimension $2\cdot 5\cdots (3n-4)$.

\medskip 

The main results of the paper may be summarized as follows.

\begin{itheorem}[Th.~\ref{th:ClassifAssDep},Th.~\ref{th:Loday}]
An operad of associator dependent algebras is Koszul if and only if it is one of the operads from the following list:
\begin{itemize}
\item The associative operad, in which all associators vanish. 
\item The operad of right pre-Lie algebras, which is the quotient by the ideal generated by 
 \[
(a_1,a_2,a_3)-(a_1,a_3,a_2), 
 \]
and the isomorphic operad of left pre-Lie algebras, which is the quotient by the ideal generated by
 \[
(a_1,a_2,a_3)-(a_2,a_1,a_3),
 \]
\item The operad of Lie-admissible algebras, which is the quotient by the ideal generated by 
 \[
(a_1,a_2,a_3)+(a_2,a_3,a_1)+(a_3,a_1,a_2)-(a_1,a_3,a_2)-(a_2,a_1,a_3)-(a_3,a_2,a_1),
 \]
\item The operad of third power associative algebras, which is the quotient by the ideal generated by 
 \[
(a_1,a_2,a_3)+(a_2,a_3,a_1)+(a_3,a_1,a_2)+(a_1,a_3,a_2)+(a_2,a_1,a_3)+(a_3,a_2,a_1),
 \]
\item Each operad in the parametric family of quotients by ideals depending on the parameter $(\alpha : \beta)\in\mathbb{P}^1$ generated by 
\begin{multline*}
\beta((a_1,a_2,a_3)+(a_2,a_1,a_3))+(\alpha-\beta)((a_1,a_3,a_2)+(a_2,a_3,a_1))\\ -\alpha((a_3,a_1,a_2)+(a_3,a_2,a_1)),
\end{multline*}
\item The magmatic operad of absolutely free nonassociative algebras; it has no relations. 
\end{itemize} 
In this list, the last three entries give an exhaustive list of Koszul operads generated by one binary operation for which the octonions form an algebra. In particular, the smallest Koszul operads that act on octonions are the operads of the parametric family.
\end{itheorem}

The classes of algebras arising in our classification are well known. Left pre-Lie algebras were discovered independently by Vinberg \cite{MR0141680} and Koszul \cite{MR145559} in the geometric context, while right pre-Lie algebras were discovered by Gerstenhaber \cite{MR161898} in the context of deformation theory. The notions of a Lie-admissible algebra and of a third power associative algebra were extensively studied by Albert \cite{MR26044,MR27750}. Most of occurrences of the parametric family in the existing literature appear when those identities are combined with the third power associativity, see, for example, \cite{MR157993,MR157994}. When $\alpha=\beta$, the corresponding algebras are known as nearly antiflexible, see \cite{MR184981}; other classes corresponding to individual values of the parameter do not seem to have been named. The earliest relevant reference that we came across is the paper \cite{MR3069613} by Max Zorn where, strikingly enough, four different identities from the above classification (the third power associativity identity, the Lie admissible identity, and the identities from the parametric family with $\alpha=\beta$ and $\alpha=-\beta$) appear on the same footing.

It would be interesting to determine the deformations of octonions in the classes of algebras listed above. We note that such deformations in some unconventional classes of algebras have already emerged in the physics literature; the celebrated Okubo algebra \cite{MR516784,MR516785} is an instance of that sort. An investigation of this matter will appear elsewhere. Another interesting direction postponed to a future paper is to advance in understanding the non-Koszul operads of associator-dependent algebras, especially the famous classes of alternative algebras, flexible algebras, various associator dependent algebras in the sense of Kleinfeld (necessarily including the third power associativity) etc.
Finally, it is natural to look for a full classification of Koszul operads generated by one binary operation; we are actively pursuing this question in an ongoing project. 

\section{Conventions and preliminary results}

All vector spaces and chain complexes are considered over a ground field $\k$ of zero characteristic. We refer the reader to \cite{MR2954392} for general information on algebraic operads and Koszul duality theory, and to \cite{MR3642294} for information on operadic Gr\"obner bases and rewriting systems.

As far as the associator dependent algebras are concerned, we shall extend the original terminology of Kleinfeld \cite{MR140547} as follows.

\begin{definition}
A vector space $V$ is called an associator dependent algebra if it is equipped with a binary product $V\otimes V\to V$, $a\otimes b\mapsto ab$, satisfying an identity 
\begin{equation}\label{eq:as-dep}
\sum_{\sigma\in S_3} x_\sigma (a_{\sigma(1)},a_{\sigma(2)},a_{\sigma(3)})=0.
\end{equation}
\end{definition}

The original definition of Kleinfeld requires that, additionally, the algebra $V$ is third power associative, that is $(xx)x=x(xx)$ for all $x\in V$, the identity first studied by Albert \cite{MR27750}. Under our assumption $\mathrm{char}(\k)=0$, this identity is equivalent to the identity 
 \[
\sum_{\sigma\in S_3} (a_{\sigma(1)},a_{\sigma(2)},a_{\sigma(3)})=0,
 \]
and in the context of classification of the corresponding operads, there is no particular reason to insist on inclusion of that identity: we shall also consider associator dependencies that do not imply third power associativity. There are several existing articles that deal with those general associator dependencies and attempt to classify them according to various properties of the corresponding algebras, see, e.g. \cite{MR2305399,MR1009578}.

In this paper, we are concerned with the Koszul property of our operads. Let us briefly recall what this means, assuming, like above, that our operads are generated by a single binary operation. To each operad $\calP$, one may associate its bar complex, the cofree cooperad co-generated by the homological shift $s\calP$ and the differential arising from the operad structure on $\calP$. An operad is said to be Koszul if the homology of the bar complex is concentrated on the ``diagonal'': for each nontrivial homology class, its degree is one less than its arity. In particular, this implies that the operad $\calP$ has only quadratic relations. For each operad with quadratic relations, the diagonal part of the bar complex has a cooperad structure which is quite easy to describe directly; that cooperad is called the Koszul dual cooperad, and is denoted $\calP^{\ac}$. Tensoring the linear dual operad of $\calP^{\ac}$ with the endomorphism operad of the homological shift $s^{-1}\k$ of the ground field, one obtains the classical Koszul dual operad $\calP^!$ of Ginzburg and Kapranov, defined as the operad whose relations annihilate the relations of $\calP$ in the certain sense, see \cite{MR1301191} and \cite[Sec.~7.6]{MR2954392}. Finally, it is useful to keep in mind that each quadratic operad $\calP$ has its associated Koszul complex $\calP^{\ac}\circ\calP$ whose homology is $\k$ if and only if $\calP$ is Koszul. 

In the language of operads, once we choose a particular identity of the form \eqref{eq:as-dep}, we may consider the operad controlling all algebras satisfying that identity. Let us make a simple but crucial observation.

\begin{proposition}\label{prop:duality}
Let $\calO$ be a quotient of the magmatic operad by the two-sided operadic ideal generated by some linear combinations of the associators of the magmatic product. The Koszul dual $\calO^!$ is a quotient of the associative operad by an ideal generated by certain ternary operations. 
\end{proposition}

\begin{proof}
For a quadratic operad generated by binary operations, all defining relations of the Koszul dual operad are ternary operations, so we only need to prove that the associative relations are among them. This follows from the well known fact that the associative operad is Koszul self-dual, so the on the dual level the $S_3$-module generated by the associativity relation annihilates any linear combination of associators, that is the associativity relations are among the relations of the Koszul dual operad. 
\end{proof}

Since an operad and its Koszul dual are either both Koszul or both non-Koszul, this proposition implies, that for the purposes of studying Koszulness one may switch freely between operads of associator dependent algebras and quotients of the associative operad. The advantage of this observation is that quotients of the associative operad by ideals generated by ternary operations are ``small'', and it is possible to explicitly compute, for each arity, the basis, dimension, and even the symmetric group action on the corresponding component, and to use that data in order to determine whether the corresponding operad is Koszul. 

The easiest and most general known way to prove the Koszul property uses Gr\"obner bases. More precisely, Gr\"obner bases for operads with symmetries are impossible to define, but it is possible to associate to each operad $\calP$ a ``shuffle operad'' whose Koszul property is equivalent to the Koszul property of $\calP$; for shuffle operads, a theory of Gr\"obner bases is available \cite[Sec.~5.4.3]{MR3642294}. For example, if we take the associative operad, as a symmetric operad it is generated by a single operation
$a_1,a_2\mapsto a_1a_2$ subject to the single relation $(a_1a_2)a_3=a_1(a_2a_3)$. In the universe of shuffle operads, one has to forget the symmetric groups actions, and write linear bases both for generators and for relations in terms of shuffle tree monomials \cite[Sec.~5.3]{MR3642294}, \cite[Sec.~8.2]{MR2954392}, which of course gives two generators and six relations in the case of the associative operad. A shuffle operad that has a quadratic Gr\"obner basis is known to be Koszul \cite[Sec.~6.4]{MR3642294}. Moreover, the same argument can be used to show that an operad presented by a convergent quadratic rewriting system \cite[Sec.~2.6]{MR3642294} is Koszul. Finding a suitable rewriting system is sometimes a matter of luck, as it heavily depends on the choice of a presentation by generators and relations. For operads generated by one binary operation, there is a useful ``polarization trick'' \cite{MR2225770} allowing sometimes to find a better presentation: it amounts to considering the generators $a_1\cdot a_2=a_1a_2+a_2a_1$ and $[a_1,a_2]=a_1a_2-a_2a_1$. In particular, for the associative operad, the polarized presentation exhibits it as the quotient by the ideal generated by the elements 
\begin{gather*}
(a_1\cdot a_2)\cdot a_3-a_1\cdot(a_2\cdot a_3)+[[a_1,a_3],a_2],\\
[a_1\cdot a_2,a_3]-[a_1,a_3]\cdot a_2-a_1\cdot[a_2,a_3].
\end{gather*}

There are also criteria of (non-)Koszulness using the Poincar\'e series, that is the exponential generating functions of the Euler characteristics of components. For an operad $\calP$ concentrated in homological degree zero, the Poincar\'e series coincides with the Hilbert series
 \[
f_{\calP}(t)=\sum_{n\ge 1}\frac{\dim\calO(n)}{n!}t^n.
 \] 
By a direct inspection, one sees that the Poincar\'e series of the Koszul complex of a quadratic operad generated by binary operations of homological degree zero is equal to $-f_{\calP^!}(-f_{\calP}(t))$. Since the Euler characteristics of a chain complex and its homology are equal, this implies that for a Koszul operad $\calP$, one has
 \[
-f_{\calP^!}(-f_{\calP}(t))=t,
 \]
so the series $f_{\calP}(t)$ and $-f_{\calP^!}(-t)$ are compositional inverses of one another. This leads to a useful positivity test of Ginzburg and Kapranov \cite{MR1301191}.   

\begin{proposition}[Positivity test]\label{prop:positivity}
Let $\calP$ be a quadratic operad generated by binary operations of homological degree zero. Denote by $a_n$ the coefficient of $t^n$ in the compositional inverse of the Poincar\'e series of that operad. If the operad $\calP$ is Koszul, then $(-1)^{n-1}a_n\ge 0$ for all $n\ge 1$.
\end{proposition}

There is also a useful sufficient condition of Koszulness in terms of Poincar\'e series; an analogous result for associative algebras is established in \cite[Cor.~2.4]{MR2177131}.

\begin{proposition}\label{prop:dim2}
Let $\calP$ be a quadratic operad generated by binary operations of homological degree zero. Suppose that $\calP(n)=0$ for $n\ge 4$, and that
 \[
-f_{\calP^!}(-f_{\calP}(t))=t.
 \]
Then the operad $\calP$ is Koszul.
\end{proposition}

\begin{proof}
Because of the condition on the operad $\calP$, the Koszul complex of the operad $\calP^!$ is concentrated in homological degrees $0$, $1$, and $2$. 
For any quadratic operad, the homology of its Koszul complex is isomorphic to $\k$ in homological degree $0$ and vanishes in homological degree $1$, so the relationship between the Poincar\'e series implies that the homology in degree $2$ also vanishes.    
\end{proof}

We note that in general neither the correct sign pattern in the coefficients of the compositional inverse nor the relationship $-f_{\calP^!}(-f_{\calP}(t))=t$ imply the Koszul property, see \cite{MR4023760,MR4106894}.

\section{Koszul quotients of the associative operad}

\subsection{Classification of relations}

The first step in understanding quotients of the associative operad by an ideal generated by ternary operations is to classify the possible ideals, that is, possible $S_3$-submodules of the arity three component of the associative operad and their consequences of higher arities. Such ideals
have been studied rather extensively. To the best of our knowledge, the first classification result (classification of generators with respect to the $S_3$-action) was given by Malcev \cite{MR0033280}. As far as the higher arity consequences are concerned, the first general study of that question was undertaken by Klein \cite{MR349741}, who came very close to a full description of all quotients by such ideals. A complete classification was accomplished by Vladimirova and Drenski in \cite{MR866654} some years later, relying crucially on previous work of Anan'in and Kemer \cite{MR0422112}. Some previous results for particular cases of modules of relations are contained in works of Dubnov and Ivanov \cite{MR0011069}, Nagata \cite{MR53088}, Higman \cite{MR73581}, Regev \cite{MR498673} and James \cite{MR491663}, listed here chronologically.

The arity three component of the associative operad is isomorphic to the regular $S_3$-module, so it decomposes as the direct sum of a copy of the one-dimensional trivial module spanned by the element 
 \[
a_1a_2a_3+a_1a_3a_2+a_2a_3a_1+a_2a_1a_3+a_3a_1a_2+a_3a_2a_1,
 \]
one copy of the one-dimensional sign module spanned by the element
 \[
a_1a_2a_3-a_1a_3a_2+a_2a_3a_1-a_2a_1a_3+a_3a_1a_2-a_3a_2a_1,
 \]
and two copies of the two-dimensional irreducible module. Informed by the approach of Anan'in and Kemer \cite{MR0422112} and of Vladimirova and Drenski \cite{MR866654}, we choose these two copies to be 
\begin{gather}
\k([a_1,a_2]a_3+[a_3,a_2]a_1, [a_1,a_3]a_2+[a_2,a_3]a_1),\label{eq:2D-1}\\
\k(a_1[a_2,a_3]+a_3[a_2,a_1], a_1[a_3,a_2]+a_2[a_3,a_1]),\label{eq:2D-2}
\end{gather}
where $[a_1,a_2]=a_1a_2-a_2a_1$ is the usual Lie bracket. Moreover, every irreducible two-dimensional $S_3$-submodule is generated by an element of the form 
\begin{equation}\label{eq:2D}
\alpha([a_1,a_2]a_3+[a_3,a_2]a_1)+\beta(a_1[a_3,a_2]+a_3[a_1,a_2])
\end{equation}
for some $(\alpha : \beta)\in\mathbb{P}^1$.

\subsection{Case-by-case study of quotients of the associative operad}

In this section, we study the Koszul property of quotients of the associative operad case by case. Since the kernel of the quotient map is generated by an $S_3$-submodule of the arity three component, there are twelve cases to consider: the multiplicity of the trivial module may be equal to $0$ or $1$, the multiplicity of the irreducible two-dimensional module may be equal to $0$, $1$, or $2$, and the multiplicity of the sign module may be equal to $0$ or $1$. A lot of our work rely on the results about dimensions of the corresponding operad obtained in \cite{MR866654}; these statements were independently verified using the existing software for computing operadic Gr\"obner bases \cite{OpGb} and the \texttt{Maple} program applying the representation theory of symmetric group to the study of polynomial identities \cite{MR3583300}. Coefficients of compositional inverses of power series were computed using \texttt{PARI/GP} \cite{PARI2} and then rechecked using \texttt{Magma} \cite{MR1484478}.

\begin{proposition}[Multiplicities $(0,0,0)$]\label{prop:first}
The quotient of the associative operad by the zero ideal is Koszul. 
\end{proposition}

\begin{proof}
This is well known, see, for example, \cite[Sec.~9.1]{MR2954392} or \cite[Sec.~5.6]{MR3642294}.
\end{proof}

\begin{proposition}[Multiplicities $(0,0,1)$]
The quotient of the associative operad by the ideal generated by  
 \[
a_1a_2a_3-a_1a_3a_2+a_2a_3a_1-a_2a_1a_3+a_3a_1a_2-a_3a_2a_1
\]
is not Koszul. 
\end{proposition}

\begin{proof}
According to \cite[Prop.~3.1]{MR866654}, the dimension of the arity $n$ component of this operad is equal to $9$ for $n=4$ and to $2n-1$ for $n\ge 5$, so the Poincar\'e series of the corresponding quotient is given by
 \[
t+t^2+\frac{5}{6}t^3+\frac{3}{8}t^4+\frac{3}{40}t^5+\frac{11}{720}t^6+O(t^7).
 \]
Its compositional inverse has positive coefficient $\frac{271}{360}$ at $t^6$. By Proposition \ref{prop:positivity}, our operad is not Koszul.
\end{proof}

\begin{proposition}[Multiplicities $(0,1,0)$]\label{prop:weight1}
For any $(\alpha : \beta)\in\mathbb{P}^1$, the quotient of the associative operad by the ideal generated by  
 \[
\alpha([a_1,a_2]a_3+[a_3,a_2]a_1)+\beta(a_1[a_3,a_2]+a_3[a_1,a_2]) 
 \]
is not Koszul.
\end{proposition}

\begin{proof}
For $\alpha\beta(\alpha-\beta)(\alpha+\beta)\ne 0$, \cite[Prop.~2.1.1]{MR866654} implies that the dimension of the arity $n$ component of this operad is equal to $1$ for $n\ge 4$, so the Poincar\'e series of our operad is given by
 \[
t+t^2+\frac{2}{3}t^3+\sum_{k\ge 4}\frac{t^k}{k!}.
 \]
Its compositional inverse has positive coefficient $\frac{461}{720}$ at $t^6$. By Proposition \ref{prop:positivity}, this operad is not Koszul. 

For $\alpha\beta=0$, \cite[Prop.~2.2.1]{MR866654} implies that the dimension of the arity $n$ component of this operad is equal to~$n$ for $n\ge 4$, so the Poincar\'e series of our operad is given by
 \[
t+t^2+\frac{2}{3}t^3+\sum_{k\ge 4}\frac{t^k}{(k-1)!}.
 \]
Its compositional inverse has a negative coefficient $-\frac{473}{720}$ at $t^7$. By Proposition \ref{prop:positivity}, this operad is not Koszul. 

For $\alpha=\beta$, \cite[Prop.~2.3.1]{MR866654} implies that the dimension of the arity $n$ component of this operad is equal to $3$ for $n=4$ and to $1$ for $n\ge 5$, so the Poincar\'e series of our operad is given by
 \[
t+t^2+\frac{2}{3}t^3+\frac{1}{8}t^4+\sum_{k\ge 5}\frac{t^k}{k!}.
 \]
Its compositional inverse has negative coefficient $-\frac{6899}{2520}$ at $t^7$. By Proposition~\ref{prop:positivity}, this operad is not Koszul.

For $\alpha=-\beta$, \cite[Prop.~2.3.1]{MR866654} implies that the dimension of the arity $n$ component of this operad is equal to $2^{n-1}$ for $n\ge 2$, so its Poincar\'e series is given by $\frac12(e^{2t}-1)$, its inverse is $\frac12\log(1+2t)$, and so this operad does not fail the positivity test. To show that it is nevertheless not Koszul, we shall use the polarized presentation, which for this operad is 
\begin{gather*}
(a_1\cdot a_2)\cdot a_3=a_1\cdot (a_2\cdot a_3),\\
[[a_1,a_2],a_3]=0,\\
[a_1\cdot a_2,a_3]=[a_1,a_3]\cdot a_2+a_1\cdot[a_2,a_3].
\end{gather*}
The advantage of this presentation is that it is homogeneous with respect to the weight grading $w(-\cdot-)=0$, $w([-,-])=1$, so our operad inherits that weight grading. It is immediate to check that $[a_1,a_2]\cdot[a_3,a_4]=[a_1,a_3]\cdot[a_2,a_4]$ follows from the defining relations, and so our operad is spanned by the elements 
 \[
a_{i_1}\cdot \cdots a_{i_s} \cdot [a_{j_1},a_{j_2}]\cdot \cdots\cdot [a_{j_{2k-1}},a_{j_{2k}}]
 \]
with $\{i_1,\cdots,i_s\}\sqcup\{j_1,\ldots,j_{2k}\}=\{1,\ldots,n\}$, which, up to a sign, depend only on the subset $\{j_1,\ldots,j_{2k}\}$ of even cardinality, see \cite[Th.~2.4.2]{MR866654}. This means that the space of elements of $\calO(n)$ of weight $k$ has dimension $\binom{n}{2k}$. 
We may now consider the weight graded version of the Poincar\'e series, associating to an operad an element of $\mathbb{Q}[u][[t]]$ for which the coefficient at $\frac{u^it^n}{n!}$ is the dimension of the weight $i$ part of the arity $n$ component. For weight graded operads with finite-dimensional graded components whose weights in each given arity are bounded, a version of Proposition \ref{prop:positivity} holds: if $a_n(u)$ is the coefficient of the compositional inverse of the Poincar\'e series of the given operad, and that operad is Koszul, then $(-1)^{n-1}a_n(u)$ is a polynomial in $u$ with non-negative coefficients for all $n\ge 1$. For our operad, the basis given above implies that the corresponding series is
 \[
t+\frac{(1+u)}2t^2+\frac{1+3u}{6}t^3+\sum_{k\ge 4}\frac{(1+u)^k+(1-u)^k}{2}\frac{t^k}{k!}.
 \]
A direct calculation shows that the coefficient at~$t^{15}$ of the compositional inverse has a negative coefficient $- \frac{53844181}{26127360}$ at $u^7$, so our operad is not Koszul.
\end{proof}

\begin{remark}
Note that in the case $\alpha=-\beta$, the polarized relations of the operad suggest that it is obtained from operads of commutative associative algebras and two-step nilpotent Lie algebras by a distributive law. This is not the case, as one can check by a direct computation, see \cite[Exercise 8.10.12]{MR2954392}. Our result show that the operad is not Koszul either, which is not \emph{a priori} clear.
\end{remark}

\begin{proposition}[Multiplicities $(0,1,1)$]
For $(\alpha : \beta)\in\mathbb{P}^1$, the quotient of the associative operad by the ideal generated by  
\begin{gather*}
\alpha([a_1,a_2]a_3+[a_3,a_2]a_1)+\beta(a_1[a_3,a_2]+a_3[a_1,a_2]),\\ 
a_1a_2a_3-a_1a_3a_2+a_2a_3a_1-a_2a_1a_3+a_3a_1a_2-a_3a_2a_1
\end{gather*}
is not Koszul if $\alpha\beta\ne0$ and is Koszul if $\alpha\beta=0$. 
\end{proposition}

\begin{proof}
For $\alpha\beta\ne 0$, it follows from \cite[Th.~2.1.2, 2.3.2, 2.4.2]{MR866654} that including the sign submodule does not bring new higher arity consequences, so the Poincar\'e series of our operad is given by
 \[
t+t^2+\frac{1}{2}t^3+\sum_{k\ge 4}\frac{t^k}{k!}.
 \]
Its compositional inverse has negative coefficient $-\frac{802543633}{39916800}$ at $t^{11}$. By Proposition \ref{prop:positivity}, this operad is not Koszul. 

For $\alpha\beta=0$, the corresponding operad is the operad of permutative algebras (associative algebras satisfying $a_1a_2a_3=a_1a_3a_2$ for $\alpha=0$ and $a_1a_2a_3=a_2a_1a_3$ for $\beta=0$); to show that it is Koszul, one may argue that it is the Koszul dual of the operad of pre-Lie algebras which is shown to be Koszul in \cite{MR1827084}.
\end{proof}

\begin{proposition}[Multiplicities $(0,2,0)$]
The quotient of the associative operad by the ideal generated by the elements $[a_1,a_2]a_3+[a_3,a_2]a_1$ and $a_1[a_3,a_2]+a_3[a_1,a_2]$
is not Koszul.
\end{proposition}

\begin{proof}
According to \cite[Th.~2.1.2]{MR866654}, the Poincar\'e series of our operad is given by
 \[
t+t^2+\frac{1}{3}t^3+\sum_{k\ge 4}\frac{t^k}{k!}.
 \]
The first 1000 coefficients of its compositional inverse have ``good'' signs, making one suspect that this operad may be Koszul. We shall however show that it is not Koszul. Like in Proposition \ref{prop:weight1}, we shall the polarized presentation , which for this operad is
\begin{gather*}
(a_1\cdot a_2)\cdot a_3=a_1\cdot (a_2\cdot a_3),\\
[[a_1,a_2],a_3]=0,\\
[a_1\cdot a_2,a_3]=[a_1\cdot a_3,a_2]=[a_1,a_2\cdot a_3]=0,\\
[a_1,a_2]\cdot a_3=-[a_1,a_3]\cdot a_2=a_1\cdot[a_2,a_3].  
\end{gather*}
Once again, for the weight grading $w(-\cdot-)=0$, $w([-,-])=1$, the relations are homogeneous, and so our operad inherits a weight grading. Clearly, the weighted Poincar\'e series of this operad is  
 \[
t+\frac{(1+u)}2t^2+\frac{1+u}{6}t^3+\sum_{k\ge 4}\frac{t^k}{k!}.
 \]
For the compositional inverse of this power series, the coefficient at~$t^{20}$ has a positive coefficient $\frac{14119421138089}{17322439680000}$ at $u^2$, so our operad is not Koszul.
\end{proof}

\begin{remark}
This operad was considered in~\cite{MR2032454}, but the non-Koszulness claim there is based on an erroneous calculations of Poincar\'e series, so  the question of its Koszulness remained open.
\end{remark}

\begin{proposition}[Multiplicities $(0,2,1)$]
The quotient of the associative operad by the ideal generated by  
\begin{gather*}
[a_1,a_2]a_3+[a_3,a_2]a_1,\qquad a_1[a_3,a_2]+a_3[a_1,a_2],\\
a_1a_2a_3-a_1a_3a_2+a_2a_3a_1-a_2a_1a_3+a_3a_1a_2-a_3a_2a_1
\end{gather*}
is Koszul.
\end{proposition}

\begin{proof}
We note that this operad is the quotient of the associative operad by the ideal generated by the elements 
 \[
a_{\sigma(1)}a_{\sigma(2)}a_{\sigma(3)}-a_1a_2a_3
 \] 
for each $\sigma\in S_3$, and the arity $n$ component of this operad is one-dimensional in each arity from three onwards. If one considers the polarized presentation of this operad, one finds that the operation $(-\cdot-)$ is associative, the operation $[-,-]$ is two-step nilpotent, and all compositions of these operations with one another vanish. This means that we are dealing with the connected sum of the operad of commutative associative algebras and the operad of anticommutative two-step nilpotent algebras. These two operads are well known to be Koszul, and so their connected sum is Koszul too (it follows from the fact that the bar complex of the connected sum is the coproduct of bar complexes). 
\end{proof}

\begin{remark}
This proof is essentially the Koszul dual of the proof explained in \cite{MR2225770}, where it is noticed that the Koszul dual operad of Lie-admissible algebras is the coproduct of the Lie operad and the commutative magmatic operad, and therefore Koszul.  
\end{remark}

\begin{proposition}[Multiplicities $(1,0,0)$]
The quotient of the associative operad by the ideal generated by  
 \[
a_1a_2a_3+a_1a_3a_2+a_2a_3a_1+a_2a_1a_3+a_3a_1a_2+a_3a_2a_1
\]
is not Koszul. 
\end{proposition}

\begin{proof}
This operad is the Koszul dual of the operad of alternative algebras, and so the theorem follows from the main result of \cite{MR2821385}.  
\end{proof}

\begin{proposition}[Multiplicities $(1,0,1)$]\label{prop:cycle}
The quotient of the associative operad by the ideal generated by 
\begin{gather*}
a_1a_2a_3+a_1a_3a_2+a_2a_3a_1+a_2a_1a_3+a_3a_1a_2+a_3a_2a_1,\\
a_1a_2a_3-a_1a_3a_2+a_2a_3a_1-a_2a_1a_3+a_3a_1a_2-a_3a_2a_1.
\end{gather*} 
is not Koszul. 
\end{proposition}

\begin{proof}
According to \cite[Th.~3.1]{MR866654}, the Poincar\'e series of this operad is given by
$t+t^2+\frac{2}{3}t^3$.
Its compositional inverse has positive coefficient $\frac{14}{9}$ at $t^6$. By Proposition \ref{prop:positivity}, our operad is not Koszul. 
\end{proof}

\begin{proposition}[Multiplicities $(1,1,0)$]
For any $(\alpha : \beta)\in\mathbb{P}^1$, the quotient of the associative operad by the ideal generated by 
\begin{gather*}
\alpha([a_1,a_2]a_3+[a_3,a_2]a_1)+\beta(a_1[a_3,a_2]+a_3[a_1,a_2]),\\ 
a_1a_2a_3+a_1a_3a_2+a_2a_3a_1+a_2a_1a_3+a_3a_1a_2+a_3a_2a_1
\end{gather*}
is not Koszul.
\end{proposition}

\begin{proof}
For $\alpha\ne-\beta$, it follows from \cite[Th.~2.1.2, Th.~2.2.2, Th.~2.3.2]{MR866654} that the Poincar\'e series of our operad is given by $t+t^2+\frac{1}{2}t^3$.
Its compositional inverse has a positive coefficient $\frac{715}{16}$ at $t^{10}$. By Proposition \ref{prop:positivity}, this operad  is not Koszul. 

For $\alpha=-\beta$, it follows from \cite[Th.~2.4.2]{MR866654} that the Poincar\'e series of our operad is given by
$t+t^2+\frac{1}{2}t^3+\frac{1}{24}t^4$.
Its compositional inverse has positive coefficient $\frac{488735}{3072}$ at $t^{12}$. By Proposition~\ref{prop:positivity}, this operad is not Koszul. 
\end{proof}

\begin{proposition}[Multiplicities $(1,1,1)$]\label{prop:111}
For any $(\alpha : \beta)\in\mathbb{P}^1$, the quotient of the associative operad by the ideal generated by 
\begin{gather*}
a_1a_2a_3+a_1a_3a_2+a_2a_3a_1+a_2a_1a_3+a_3a_1a_2+a_3a_2a_1,\\
a_1a_2a_3-a_1a_3a_2+a_2a_3a_1-a_2a_1a_3+a_3a_1a_2-a_3a_2a_1,\\
\alpha([a_1,a_2]a_3+[a_3,a_2]a_1)+\beta(a_1[a_3,a_2]+a_3[a_1,a_2])
\end{gather*}
is Koszul.
\end{proposition}

\begin{proof}
Throughout the proof, we shall denote this operad $\calO_{\alpha,\beta}$. 
It follows from \cite[Th.~2.1.2, Th.~2.2.2, Th.~2.3.2, Th.~2.4.2]{MR866654} that the Poincar\'e series of the operad $\calO_{\alpha,\beta}$ does not depend on $(\alpha : \beta)$ and is given by
$t+t^2+\frac{1}{3}t^3$.
In the polarized presentation, the generators of the ideal of relations of our operad are
\begin{gather*}
(a_1\cdot a_2)\cdot a_3-a_1\cdot(a_2\cdot a_3)+[[a_1,a_3],a_2],\\
[a_1\cdot a_2,a_3]-[a_1,a_3]\cdot a_2-a_1\cdot[a_2,a_3],\\
(a_1\cdot a_2)\cdot a_3+(a_3\cdot a_1)\cdot a_2+(a_2\cdot a_3)\cdot a_1,\\
[a_1,a_2]\cdot a_3+[a_3,a_1]\cdot a_2+a_1\cdot [a_2,a_3],\\
(\alpha-\beta)([[a_1,a_2]a_3]+[[a_3,a_2]a_1])+(\alpha+\beta)(a_1\cdot [a_3,a_2]+a_3\cdot [a_1,a_2]).
\end{gather*}

\noindent
The proof of the Koszul property depends on $(\alpha : \beta)$. Suppose first that $\alpha=-\beta$. In this case, the polarized presentation of the operad $\calO_{\alpha,\beta}$ may be simplified to
\begin{gather*}
(a_1\cdot a_2)\cdot a_3=[[a_1,a_2],a_3]=0,\\
[a_1\cdot a_2,a_3]-[a_1,a_3]\cdot a_2-[a_2,a_3]\cdot a_1,\\
a_1\cdot[a_2,a_3]+a_2\cdot[a_3,a_1]+a_3\cdot [a_1,a_2].
\end{gather*}
The associated shuffle operad may be determined by a convergent quadratic rewriting system
\begin{gather*}
(a_1\cdot a_2)\cdot a_3\to0,\qquad [[a_1,a_2],a_3]\to0,\\
(a_1\cdot a_3)\cdot a_2\to0,\qquad [[a_1,a_3],a_2]\to0,\\
a_1\cdot (a_2\cdot a_3)\to0,\qquad [a_1,[a_2,a_3]]\to0,\\
[a_1\cdot a_2,a_3]\to [a_1,a_3]\cdot a_2+a_1\cdot [a_2,a_3],\\
[a_1\cdot a_3,a_2]\to[a_1,a_2]\cdot a_3-a_1\cdot [a_2,a_3],\\
[a_1,a_2\cdot a_3]\to[a_1,a_2]\cdot a_3+[a_1,a_3]\cdot a_2,\\
a_1\cdot[a_2,a_3]\to[a_1,a_3]\cdot a_2-[a_1,a_2]\cdot a_3.
\end{gather*}
Termination follows from the fact that with each application of the rewriting rule we either rewrite elements into zero, or put the operation $(-\cdot -)$ closer to the root of the tree, or increase the associated path sequence (the number of shuffle tree monomials of the given arity is finite, so increasing is as good as decreasing). Confluence follows from the fact that the ternary normal forms are $[a_1,a_3]\cdot a_2$ and $[a_1,a_2]\cdot a_3$, so there are no normal forms of higher arities, and we have the expected dimensions of the components of the operad. Consequently, our operad is Koszul. 

For other values of the parameter, there is no convergent quadratic rewriting system for our operad, and the argument will be more intricate: we shall study the Koszul dual operad, and show that Proposition \ref{prop:dim2} is applicable. The Koszul dual operad $\calO_{\alpha,\beta}^!$ describes algebras with the following relation between associators:
\begin{multline*}
\beta((a_1,a_2,a_3)+(a_2,a_1,a_3))+(\alpha-\beta)((a_1,a_3,a_2)+(a_2,a_3,a_1))\\ -\alpha((a_3,a_1,a_2)+(a_3,a_2,a_1))=0.
\end{multline*}
To see that, we first note that, according to Proposition \ref{prop:duality}, the operad $\calO_{\alpha,\beta}^!$ is some operad of associator dependent algebras. 
Moreover, the annihilator of the four-dimensional space of relations is two-dimensional, and it is clear that this two-dimensional module is irreducible. Every two-dimensional submodule of associator dependencies is generated by an element of the form
\begin{multline*}
\lambda((a_1,a_2,a_3)+(a_2,a_1,a_3))-(\lambda+\mu)((a_1,a_3,a_2)+(a_2,a_3,a_1)) \\
 +\mu((a_3,a_1,a_2)+(a_3,a_2,a_1)),
\end{multline*}
for some $(\lambda : \mu)\in\mathbb{P}^1$. Computing the pairing between this element and the defining relations of $\calO_{\alpha,\beta}$, we get $\lambda\alpha+\beta\mu=0$, so $(\lambda : \mu)=(-\beta : \alpha)$. 

For the operad $\calO_{\alpha,\beta}^!$, we shall also make use of the polarized presentation, for which the ideal of relations of the associated shuffle operad is generated by
\begin{multline*}
(\alpha+\beta)([[a_1,a_2],a_3]+[a_1,[a_2,a_3]]-(a_1\cdot a_2)\cdot a_3+2(a_1\cdot a_3)\cdot a_2-a_1(a_2\cdot a_3))+\\
(\alpha-\beta)(-[a_1\cdot a_2,a_3]+2[a_1\cdot a_3,a_2]+[a_1,a_2\cdot a_3]-3[a_1,a_2]\cdot a_3+3a_1\cdot[a_2,a_3])
\end{multline*}
and
\begin{multline*}
(\alpha+\beta)([[a_1,a_3],a_2]-[a_1,[a_2,a_3]]+2(a_1\cdot a_2)\cdot a_3-(a_1\cdot a_3)\cdot a_2-a_1(a_2\cdot a_3))+\\
(\alpha-\beta)(2[a_1\cdot a_2,a_3]-[a_1\cdot a_3,a_2]+[a_1,a_2\cdot a_3]-3[a_1,a_3]\cdot a_2-3a_1\cdot[a_2,a_3]).
\end{multline*}

Suppose first that $\alpha=\beta$, that is $(\alpha : \beta)=(1:1)$. In this case, the polarized presentation simplifies, and the ideal of relations of the associated shuffle operad  is generated by
\begin{gather}
[[a_1,a_2],a_3]+[a_1,[a_2,a_3]]-(a_1\cdot a_2)\cdot a_3+2(a_1\cdot a_3)\cdot a_2-a_1(a_2\cdot a_3),\\
[[a_1,a_3],a_2]-[a_1,[a_2,a_3]]+2(a_1\cdot a_2)\cdot a_3-(a_1\cdot a_3)\cdot a_2-a_1(a_2\cdot a_3).
\end{gather}
Let us consider the graded path-lexicographic ordering with $[-,-]>(-\cdot-)$. It turns out that for this ordering, the operad $\calO_{1,1}^!$ has a finite Gr\"obner basis consisting of the above quadratic elements with the leading terms $[[a_1,a_2],a_3]$ and $[[a_1,a_3],a_2]$ and six cubic elements whose leading terms are $[((a_1\cdot a_i)\cdot a_j,a_k]$ for all possible permutations $(i,j,k)$ of $\{2,3,4\}$. It follows from \cite{MR3084563} that our operad has a model (free resolution) whose generators correspond to overlaps of the leading terms. For each arity $n\ge 4$, there are $(n-1)!$ such generators that are overlaps of the first two leading terms, and $(n-1)!$ such generators that are overlaps involving the other leading terms. These generators have opposite parities, contributing zero to the Euler characteristics, so the Poincar\'e series of the space of generators is equal to $t-t^2+\frac13 t^3$. Since the Poincar\'e series of the space of generators of a model of an operad is always equal to the compositional inverse of the Poincar\'e series of an operad, Proposition \ref{prop:dim2} means that the operad $\calO_{1,1}$ is Koszul.

Let us now consider the case $\alpha\ne\beta$. We shall show that the Poincar\'e series of the operad $\calO_{\alpha,\beta}^!$ is still equal to the compositional inverse of $t-t^2+\frac13t^3$, which, comnbined with Proposition \ref{prop:dim2}, will prove the Koszul property. However, the proof is going to be more intricate, since for most values of parameter there is no convergent quadratic rewriting system for the Koszul dual operad, and for some values of parameter there is no known convergent rewriting system at all. Thus, we shall show separately that the coefficients of the compositional inverse give an upper and a lower bound for the dimensions of components. 

\emph{Lower bound. } Let us denote $s=\frac{\alpha+\beta}{\alpha-\beta}$. Recall that the operad $\calO_{\alpha,\beta}^!$ is, up to homological shifts and linear duality, the diagonal part of the bar complex of the operad $\calO_{\alpha,\beta}$, so for the purposes of estimating the dimensions of components, we may focus on the latter chain complex. We know that for all values of $(\alpha : \beta)$, all components of the operad $\calO_{\alpha,\beta}$ starting from the arity four vanish. Moreover, a direct inspection of the polarized presentation of the operad $\calO_{\alpha,\beta}$ shows that for $\alpha\ne\beta$, the cosets of the elements $[a_1,a_2]\cdot a_3$ and $[a_1,a_3]\cdot a_2$ form a basis in the component $\calO_{\alpha,\beta}(3)$, and the structure constants expressing compositions of generators as combinations of these elements are polynomials in $s$. Let us fix an arity $n\ge 1$. What we just said implies that the arity $n$ component of the bar complex of that operad is a chain complex of flat $\k[s]$-modules of finite rank, hence the semicontinuity theorem \cite[Sec.~III.12]{MR0463157} applies, and for each integer $k\ge 0$, the $k$-th homology of this chain complex is constant for generic $s$, and may jump up for certain special values of $s$. Therefore,
\begin{itemize}
\item for generic values of $(\alpha : \beta)$, the homology of first $n$ arities of the bar complex of the operad $\calO_{\alpha,\beta}$ is concentrated on the diagonal (since the off-diagonal homology groups vanish for one specialisation $s=0$, corresponding to $\alpha=-\beta$, and since homology is semicontinuous),
\item for generic values of $(\alpha : \beta)$, the first $n$ coefficients of the Poincar\'e series of the operad $\calO_{\alpha,\beta}^!$ are  equal to the first $n$ coefficients of the compositional inverse of $t-t^2+\frac13t^3$ (since the Poincar\'e series of the bar complex of an operad  is always equal to the compositional inverse of the Poincar\'e series of that operad, and since we already know that for generic values, the homology of the first $n$ arities of the bar complex of the operad $\calO_{\alpha,\beta}$ is concentrated on the diagonal),
\item for each value of $(\alpha : \beta)$, the $n$-th coefficient of the Poincar\'e series of the operad $\calO_{\alpha,\beta}^!$ is greater than or equal to the $n$ coefficient of the compositional inverse of $t-t^2+\frac13t^3$ (since homology is semicontinuous),
\end{itemize}
so the compositional inverse of $t-t^2+\frac13t^3$ is a lower bound for all possible Poincar\'e series. 

\emph{Upper bound. } We shall show that the shuffle tree monomials whose underlying planar trees are binary, whose vertices are labelled by the polarized operations $(-\cdot-)$ and $[-,-]$, and whose quadratic divisors do not include $[a_1,a_2]\cdot a_3$ and $[a_1,a_3]\cdot a_2$ span the Koszul dual operad. Unfortunately, one can show that there exists no convergent rewriting system with these monomials as normal forms, so we shall use some sort of rewriting that converges but does not have a direct meaning in terms of operads. Let us prove the spanning property as follows. Overall, we shall argue by induction on arity. For fixed arity, we shall argue taking into account the label of the root vertex. Let $T$ be any shuffle tree monomial of arity $n$. If the root of $T$ is labelled by $[-,-]$, then, once we use the induction hypothesis and represent the two trees grafted at the root of $T$ as linear combinations of requested shuffle tree monomials, this immediately gives such a representation of $T$. Suppose that the root of $T$ is labelled by $(-\cdot-)$, so that it may have a left quadratic divisor at the root that is prohibited. Since $\alpha\ne\beta$, the two prohibited quadratic divisors appear (individually) in the two defining relations of our operad, and we may replace the arising divisor by a linear combination of allowed quadratic monomials. In the result, we may forget the monomials where the root is labelled by $[-,-]$, since we already proved our statement for such monomials. What are the other monomials that may appear? Among them there are monomials which have fewer occurrences of $[-,-]$, which we may make another induction parameter, and monomials which have the same number of occurrences of $[-,-]$, but the arity of the left subtree of the root is smaller, which we may make another induction parameter. This means that it is possible to write $T$ as a linear combination of requested shuffle tree monomials. We already know that these monomials form a basis in the Koszul dual operad for $\alpha=\beta$, so the necessary upper bound is established. 

Combining the two bounds that we found, we conclude that the Poincar\'e series of our operad is the compositional inverse of $t-t^2+\frac13t^3$, so according to Proposition \ref{prop:dim2}, our operad is Koszul.
\end{proof}

\begin{remark}
One can adapt the above proof for the case $\alpha=\beta$ to most cases: for $\alpha^2-\alpha+1\ne0$, the polarized presentation of the Koszul dual operad leads to a convergent rewriting system with quadratic and cubic right hand sides, for which there is the ``correct'' number of normal forms. However, the exceptional case $\alpha^2-\alpha+1=0$, which, incidentally, corresponds to the rather elegant identity 
 \[
(a_1,a_2,a_3)+\omega(a_2,a_3,a_1)+\omega^2(a_3,a_1,a_2)=0,
 \] 
where $\omega=-\alpha$ is a primitive third root of unity, does not seem to admit a finite operadic rewriting system, and so one has to resort to the strategy described above, where an upper bound is obtained in a different way. 
\end{remark}

\begin{proposition}[Multiplicities $(1,2,0)$]
The quotient of the associative operad by the ideal generated by  
\begin{gather*}
[a_1,a_2]a_3+[a_3,a_2]a_1,\\
a_1[a_3,a_2]+a_3[a_1,a_2],\\
a_1a_2a_3+a_1a_3a_2+a_2a_3a_1+a_2a_1a_3+a_3a_1a_2+a_3a_2a_1
\end{gather*}
is Koszul.
\end{proposition}

\begin{proof}
We note that this operad is the quotient of the associative operad by the ideal generated by the elements 
 \[
a_{\sigma(1)}a_{\sigma(2)}a_{\sigma(3)}-(-1)^{\sigma}a_1a_2a_3
 \] 
for each $\sigma\in S_3$. Thus, each coset of the basis elements of the associative operad may be taken as a basis element of its ternary component, and all components of arity at least four vanish. Once again, it is advantageous to use the polarized presentation which for our operad is
\begin{gather*}
(a_1\cdot a_2)\cdot a_3=[a_1\cdot a_2, a_3]=[[a_1,a_2],a_3]=0,\\
[a_1,a_2]\cdot a_3+[a_1,a_3]\cdot a_2=0.
\end{gather*}
If we consider the graded path-lexicographic ordering (for any ordering of generators), the only normal form of arity $3$ is the element $a_1\cdot[a_2,a_3]$, so there are no normal forms of arity four, and the elements listed above form a quadratic Gr\"obner basis. Consequently, our operad is Koszul. 
\end{proof}

\begin{remark}
We note that this operad was considered in \cite{MR3642246}, where it is erroneusly claimed that it is not Koszul.
\end{remark}

\begin{proposition}[Multiplicities $(1,2,1)$]\label{prop:last}
The quotient of the associative operad by its arity three component is Koszul.
\end{proposition}

\begin{proof}
This is the operad of nilpotent algebras of index three which is well known to be Koszul. 
\end{proof}

\subsection*{Classification theorem}
The results we obtained in Propositions \ref{prop:first}--\ref{prop:last} prove the following result, which is one of the main classification results of this paper. 

\begin{theorem}\label{th:KoszulQuot}
Koszul quotients of the associative operad by an ideal generated by ternary operations are precisely the operads from the following list:
\begin{itemize}
\item The associative operad, which is the quotient by the zero ideal. 
\item The operad of right permutative algebras, which is the quotient by the ideal generated by 
 \[
a_1a_2a_3-a_1a_3a_2,
 \]
and the isomorphic operad of left permutative algebras, which is the quotient by the ideal generated by
 \[
a_1a_2a_3-a_2a_1a_3,
 \]
\item The operad of bipermutative algebras, which is the quotient by the ideal generated by 
\begin{gather*}
a_1a_2a_3-a_1a_3a_2,\\
a_1a_2a_3-a_2a_1a_3,
\end{gather*} 
\item The operad of biantipermutative algebras, which is the quotient by the ideal generated by
\begin{gather*}
a_1a_2a_3+a_1a_3a_2,\\
a_1a_2a_3+a_2a_1a_3
\end{gather*}
\item Each operad in the parametric family of quotients by ideals depending on the parameter $(\alpha : \beta)\in\mathbb{P}^1$ generated by 
\begin{gather*}
a_1a_2a_3+a_2a_3a_1+a_3a_1a_2,\\
\alpha([a_1,a_2]a_3+[a_3,a_2]a_1)+\beta(a_1[a_3,a_2]+a_3[a_1,a_2]),
\end{gather*}
\item The operad of associative algebras that are nilpotent of index three, which is the quotient by the ideal generated by $a_1a_2a_3$. 
\end{itemize} 
\end{theorem}

\begin{proof}
Most of the claims of the theorem are proved above. The claim that the ideals in the parametric family of operads are generated by
\begin{gather*}
a_1a_2a_3+a_2a_3a_1+a_3a_1a_2,\\
\alpha([a_1,a_2]a_3+[a_3,a_2]a_1)+\beta(a_1[a_3,a_2]+a_3[a_1,a_2])
\end{gather*}
follows from the fact that the submodule generated by $a_1a_2a_3+a_2a_3a_1+a_3a_1a_2$ is precisely the sum of the trivial and the sign submodules. 
\end{proof}

\subsection{Koszul quotients by binary and ternary operations}

If we consider arbitrary Koszul quotients of the associative operad, one must also look at the case in which the kernel has a non-zero intersection with the two-dimensional regular $S_2$-module of generators. In such case, there are several possible case. First, the intersection of the kernel with the module of generators may be equal to the whole module of generators, in which case the only quotient is the unit operad. Next, the intersection of the kernel with the module of generators may coincide with the sign module, in which case there are two quadratic quotients, the operad of commutative associative algebras which is well known to be Koszul, and the operad of two-step nilpotent commutative associative algebras which is also Koszul (it is the Koszul dual of a free operad). Finally, the intersection of the kernel with the module of generators may coincide with the trivial module, in which case there are two quadratic quotients, the operad of anti-commutative associative algebras which is well known to not be Koszul (for it fails the positivity test), and the operad of two-step nilpotent anti-commutative associative algebras which is also Koszul (it is the Koszul dual of a free operad). 

\section{Koszul operads of associator dependent algebras}

\subsection{Main theorem}

We are now ready to answer the original question. The following theorem is the main result of this paper.

\begin{theorem}\label{th:ClassifAssDep}
An operad of associator dependent algebras is Koszul if and only if it is one of the operads from the following list:
\begin{itemize}
\item The associative operad, in which all associators vanish. 
\item The operad of right pre-Lie algebras, which is the quotient by the ideal generated by 
 \[
(a_1,a_2,a_3)-(a_1,a_3,a_2), 
 \]
and the isomorphic operad of left pre-Lie algebras, which is the quotient by the ideal generated by
 \[
(a_1,a_2,a_3)-(a_2,a_1,a_3),
 \]
\item The operad of Lie-admissible algebras, which is the quotient by the ideal generated by 
 \[
(a_1,a_2,a_3)+(a_2,a_3,a_1)+(a_3,a_1,a_2)-(a_1,a_3,a_2)-(a_2,a_1,a_3)-(a_3,a_2,a_1),
 \]
\item The operad of third power associative algebras, which is the quotient by the ideal generated by 
 \[
(a_1,a_2,a_3)+(a_2,a_3,a_1)+(a_3,a_1,a_2)+(a_1,a_3,a_2)+(a_2,a_1,a_3)+(a_3,a_2,a_1),
 \]
\item Each operad in the parametric family of quotients by ideals depending on the parameter $(\alpha : \beta)\in\mathbb{P}^1$ generated by 
\begin{multline*}
\beta((a_1,a_2,a_3)+(a_2,a_1,a_3))+(\alpha-\beta)((a_1,a_3,a_2)+(a_2,a_3,a_1))\\ -\alpha((a_3,a_1,a_2)+(a_3,a_2,a_1)),
\end{multline*}
\item The magmatic operad of absolutely free nonassociative algebras. 
\end{itemize} 
\end{theorem}

\begin{proof}
It is well known \cite{MR1827084} that the operad of right pre-Lie algebras is the Koszul dual of the operad of right permutative algebras, and same applies to the left versions of those operads, which control the opposite algebras. It is also known \cite{MR2032454,MR3642246} that the Koszul dual of the operad of Lie admissible algebras is the operad of bipermutative algebras, and that the Koszul of the operad of third power associative algebras is the operad of biantipermutative algebras. For the parametric family, the Koszul dual was computed in the proof of Proposition \ref{prop:111}.
\end{proof}

Let us use the Koszul property to obtain information on the Poincar\'e series of operads of associator dependent algebras. In the case of the operad of pre-Lie algebras, it is well known \cite{MR1827084} that the arity $n$ component is of dimension $n^{n-1}$. The Koszul property of the operads belonging to the parametric family implies the following explicit formula, which it is amusing to compare with the dimension formula $1\cdot 2\cdots n$ for the arity $n$ component of the associative operad and the dimension formula $2\cdot 6\cdots (4n-6)$ for the magmatic operad. 

\begin{corollary}\label{cor:param}
For each operad in the parametric family of quotients by ideals depending on the parameter $(\alpha : \beta)\in\mathbb{P}^1$ generated by 
\begin{multline*}
\beta((a_1,a_2,a_3)+(a_2,a_1,a_3))+(\alpha-\beta)((a_1,a_3,a_2)+(a_2,a_3,a_1))\\ -\alpha((a_3,a_1,a_2)+(a_3,a_2,a_1)),
\end{multline*}
the dimension of the arity $n$ component is equal to
 \[
2\cdot 5\cdots (3n-4).
 \]
\end{corollary}

\begin{proof}
It follows from Proposition \ref{prop:111} that this operad is Koszul, and that the Poincar\'e series of the Koszul dual operad is 
 \[
f(t)=t+t^2+\frac{t^3}3=-\frac13+\frac13(1+t)^3.
 \] 
The compositional inverse of $-f(-t)$ is given by
\begin{multline*}
1-(1-3t)^{1/3}=1-\sum_{n\ge 0}(-3t)^n\binom{1/3}n=\\
1-\sum_{n\ge 0}\frac{1(-2)\cdots (1-3(n-1))}{n!}(-t)^n=
t+\sum_{n\ge 2}\frac{2\cdot5\cdots (3n-4)}{n!}t^n,
\end{multline*}
completing the proof.
\end{proof}

For the other Koszul operads of associator dependent algebras, there are no closed formulas for dimensions of components. We record the following consequences of their Koszulness. 

\begin{proposition}\leavevmode
\begin{itemize}
\item The Poincar\'e series $f(t)$ of the operad of third power associative algebras satisfies the equation
 \[
f(t)-f(t)^2+\frac{f(t)^3}{6}=t. 
 \]
\item The Poincar\'e series $f(t)$ of the operad of Lie-admissible algebras satisfies the equation
 \[
1-e^{-f(t)}-\frac{f(t)^2}{2}=t.
 \]
\end{itemize}
\end{proposition}

We note that in the first case, the same equation is satisfied by the Koszul operad of anti-Lie-admissible (alia) algebras \cite{MR2676255}; this may be explained by observing that the corresponding shuffle operads have quadratic Gr\"obner bases~\cite{MR3301915}, and the combinatorics of leading monomials of the two Gr\"obner bases are the same. The second statement was conjectured in \cite{oeis1}. 

\subsection{Application to the question of Loday}

From Theorem \ref{th:ClassifAssDep}, we immediately obtain the following result. 

\begin{theorem}\label{th:Loday}
A Koszul operad with one binary generator for which the octonions form an algebra is one of the following operads:
\begin{itemize}
\item the operad of third power associative algebras,
\item each operad in the parametric family of quotients by ideals depending on the parameter $(\alpha : \beta)\in\mathbb{P}^1$ generated by 
\begin{multline*}
\beta((a_1,a_2,a_3)+(a_2,a_1,a_3))+(\alpha-\beta)((a_1,a_3,a_2)+(a_2,a_3,a_1))\\ -\alpha((a_3,a_1,a_2)+(a_3,a_2,a_1)),
\end{multline*}
\item the magmatic operad.
\end{itemize}
In particular, the smallest Koszul operads that act on octonions are the operads of the parametric family.
\end{theorem}

\begin{proof}
First, we note that the product of octonions is neither commutative nor anticommutative, so a binary operad acting on the algebra of octonions is generated by an operation without any symmetries. Next, we recall that the algebra of octonions is alternative, so the alternative operad should admit a map from the operad we are interested in. Moreover, the product of octonions does not satisfy identities of arity less than five that do not follow from alternativity \cite{MR937613}, and a Koszul operad is quadratic, so there are no further conditions. 
The $S_3$-module of the alternativity relations is a direct sum of one copy of the trivial module and two copies of the two-dimensional irreducible module, so the only condition on our operad is that its module of relation does not contain the sign module. Examining the classification result, we see that this rules out the associative, the pre-Lie, and the Lie-admissible case. The last claim comes from examining the Poincar\'e series computed above.   
\end{proof}

One can extend the question of Loday as follows. In the chain of algebras $\mathbb{R}$, $\mathbb{C}$, $\mathbb{H}$, and $\mathbb{O}$, each one is obtained from the previous one by the so called Cayley--Dickson process. One may not stop at octonions, obtaining further an algebra of the so called sedenions and further algebras that do not seem to have names; these algebras will still be normed, but they will not be division algebras anymore. The following result classifies Koszul operads that act on those algebras.

\begin{theorem}
A Koszul operad with one binary generator that acts on all algebras obtained from $\k$ by the Cayley--Dickson process is one of the following operads:
\begin{itemize}
\item the operad of third power associative algebras,
\item the operad that is the quotient of the magmatic operad by the ideal generated by 
 \[
(a_1,a_2,a_3)+(a_2,a_1,a_3)-2(a_1,a_3,a_2)-2(a_2,a_3,a_1)+(a_3,a_1,a_2)+(a_3,a_2,a_1),
 \]
\item the magmatic operad.
\end{itemize}
\end{theorem}

\begin{proof}
From the results of the first author and Hentzel \cite{MR1849501}, it follows that all ternary identity of sedenions follow from the flexible law 
 \[
(a_1,a_2,a_3)+(a_3,a_2,a_1)=0,
 \]
and that each such identity is satisfied by all the algebras obtained from the Cayley--Dickson process if and only if it is satisfied by the sedenions. Thus, it remains to determine which of the operads from the previous theorem map to the operad of flexible algebras. The $S_3$-module of the flexibility relations is a direct sum of one copy of the trivial module and the two-dimensional irreducible module, so the possible submodules are the zero module, the trivial module, and the two-dimensional irreducible module (one particular choice, rather than any one from the parametric family). These corresponds to the quotients being the magmatic operad, the third power associative operad, and the particular operad from the parametric family (corresponding to $\alpha=-\beta$). 
\end{proof}

\section*{Acknowledgments}

To the best of our knowledge, the first person to have studied the class of associator-dependent algebras as a whole was Erwin Kleinfeld who passed away in January 2022. His work on various classes of nonassociative algebras has been influencing researchers for many decades, and we wish to dedicate this paper to his memory. We thank Fatemeh Bagherzadeh for suggesting this problem to us, and for her involvement in this project at its early stages. The second author is grateful to Ivan Pavlovich Shestakov for several interesting discussions of some aspects of this project. We thank Frederic Chapoton, Dmitri Piontkovski, and Pedro Tamaroff for comments on a preliminary version of the article.

\noindent
The first author was supported by the NSERC grant ``Algebraic Operads''. The second author was supported by Institut Universitaire de France, by the University of Strasbourg Institute for Advanced Study (USIAS) through the Fellowship USIAS-2021-061 within the French national program ``Investment for the future'' (IdEx-Unistra), and by the French national research agency project ANR-20-CE40-0016. Some of the final breakthroughs in the main result of the paper were made during the second author's stay at Max Planck Institute for Mathematics in Bonn, and he wishes to express his gratitude to that institution for the financial support and excellent working conditions.

\bibliographystyle{plain}
\bibliography{biblio}

\end{document}